\documentclass[10pt,letterpaper,dvips]{article}

 \usepackage{amssymb,latexsym,amsmath,amsthm}
\newtheorem{thm*}{Theorem}

\newtheorem{thm}{Theorem}[section]

\newtheorem{dfn}{Definition}[section]

\newtheorem{lemma}{Lemma}[section]

\newtheorem{remark}{Remark}[section]

\newtheorem{cor}{Corollary}[section]

\usepackage{fullpage}

\begin{document}

\def\d{ \partial } 
\def\Na{{\mathbb{N}}}

\def\Z{{\mathbb{Z}}}

\def\IR{{\mathbb{R}}}

\def\L{ {\mathcal{L}}}

\newcommand{\E}[0]{ \varepsilon}

\newcommand{\s}[0]{ \mathcal{S}}

\newcommand{\AO}[1]{\| #1 \| }

\newcommand{\BO}[2]{ \left( #1 , #2 \right) }

\newcommand{\CO}[2]{ \left\langle #1 , #2 \right\rangle}

\newcommand{\co}[1]{ #1^{\prime}} 

\newcommand{\p}[0]{ p^{\prime}} 

\newcommand{\m}[1]{   \mathcal{ #1 }}

\newcommand{ \A}[1]{ \left\| #1 \right\|_H }

\newcommand{\B}[2]{ \left( #1 , #2 \right)_H }

\newcommand{\C}[2]{ \left\langle #1 , #2 \right\rangle_{  H^* , H } }

 \newcommand{\HON}[1]{ \| #1 \|_{ H^1} }

\newcommand{ \Om }{ \Omega}

\newcommand{ \pOm}{\partial \Omega}

\newcommand{\D}{ \mathcal{D} \left( \Omega \right)} 

\newcommand{\Ov}{ \overline{ \Omega}}

\newcommand{\DP}{ \mathcal{D}^{\prime} \left( \Omega \right)  }

\newcommand{\DPP}[2]{   \left\langle #1 , #2 \right\rangle_{  \mathcal{D}^{\prime}, \mathcal{D} }}

\newcommand{\PHH}[2]{    \left\langle #1 , #2 \right\rangle_{    \left(H^1 \right)^*  ,  H^1   }    }

\newcommand{\PHO}[2]{  \left\langle #1 , #2 \right\rangle_{  H^{-1}  , H_0^1  }} 

 \newcommand{\HO}{ H^1 \left( \Omega \right)}

\newcommand{\HOO}{ H_0^1 \left( \Omega \right) }

\newcommand{\CC}{C_c^\infty\left(\Omega \right) }

\newcommand{\N}[1]{ \left\| #1\right\|_{ H_0^1  }  }

\newcommand{\IN}[2]{ \left(#1,#2\right)_{  H_0^1} }

\newcommand{\INI}[2]{ \left( #1 ,#2 \right)_ { H^1}} 

\newcommand{\HH}{   H^1 \left( \Omega \right)^* } 

\newcommand{\HL}{ H^{-1} \left( \Omega \right) }

\newcommand{\HS}[1]{ \| #1 \|_{H^*}}

\newcommand{\HSI}[2]{ \left( #1 , #2 \right)_{ H^*}}

\newcommand{\WO}{ W_0^{1,p}} 
\newcommand{\w}[1]{ \| #1 \|_{W_0^{1,p}}}  

\newcommand{\ww}{(W_0^{1,p})^*}

\newcommand{\labeld}[1]{ \mbox{ \qquad (#1)}  \qquad   \label{ #1}  }

\newcommand{\bi}{\Delta^2}
\newcommand{\la}{\lambda}
\newcommand{\f}{\frac}
\newcommand{\wt}{\widetilde}
\newcommand{\dis}{\displaystyle}
\newcommand{\om}{\Omega}

\title{The critical dimension for a fourth order elliptic problem with singular nonlinearity}

\author{\sc{Craig COWAN\footnote{Department of Mathematics, University of British Columbia, Vancouver, B.C. Canada V6T 1Z2. E-mail: cowan@math.ubc.ca. }
\quad
Pierpaolo ESPOSITO\footnote{Dipartimento di Matematica,
Universit\`a degli Studi ``Roma Tre", 00146 Roma, Italy. E-mail: esposito@mat.uniroma3.it. Research supported by M.U.R.S.T., project ``Variational methods and
nonlinear differential equations".}\quad
Nassif GHOUSSOUB\footnote{Department of Mathematics, University of British Columbia, Vancouver, B.C. Canada V6T 1Z2. E-mail: nassif@math.ubc.ca. Research partially supported by the Natural Science and Engineering Research Council of Canada.}\quad
}}
\date{\today}

\smallbreak
\maketitle

\begin{abstract} We study the regularity of the extremal solution of the semilinear biharmonic equation $\bi u=\f{\lambda}{(1-u)^2}$, which models a simple Micro-Electromechanical System (MEMS) device on  a ball $B\subset\IR^N$, under Dirichlet boundary conditions
$u=\partial_\nu u=0$ on $\partial B$.  We complete here the results of F.H. Lin and Y.S. Yang \cite{LY}  regarding the identification of a ``pull-in voltage"  $\la^*>0$ such that  a  stable classical solution $u_\la$ with $0<u_\la<1$ exists for $\la\in (0,\la^*)$, while there is none of any kind when $\la>\la^*$. Our main result asserts  that  the extremal solution $u_{\lambda^*}$ is regular $(\sup_B u_{\lambda^*}  <1)$ provided $ N \le 8$ while  $u_{\lambda^*} $ is singular ($\sup_B u_{\lambda^*} =1$) for $N \ge 17$, in which case $1-C_0|x|^{4/3}\leq u_{\lambda^*} (x) \leq  1-|x|^{4/3}$ on the unit ball, where $ C_0:= \left( \frac{\lambda^*}{\overline{\lambda}}\right)^\frac{1}{3}$ and $ \bar{\lambda}:= \frac{8 (N-\frac{2}{3}) (N- \frac{8}{3})}{9}$. The singular character of the extremal solution for the remaining cases (i.e., when $9\leq N\leq 16$) requires a computer assisted proof and will not be addressed in this paper. 
\end{abstract}

\section{Introduction}   
The following model has been proposed  for the  description of  the steady-state of a simple Electrostatic MEMS device:
\begin{equation}
\label{MEMS.1}
 \arraycolsep=1.5pt 
\left\{ \begin{array}{ll} 
\alpha \Delta^2 u = \left( \beta \int_\Omega | \nabla u|^2 dx + \gamma \right) \Delta u + \frac{ \lambda f(x)}{ (1-u)^2 \left( 1 + \chi \int_\Omega \frac{dx}{(1-u)^2} \right)}  &\quad \hbox{in }\Omega \\
0<u<1      &\quad \hbox{in } \Omega \\
u=\alpha \partial_\nu u =0  &\quad \hbox{on } \partial \Omega , \quad
\end{array} \right. \end{equation}
 where $ \alpha, \beta, \gamma, \chi \ge 0$, $ f \in C( \overline{\Omega},[0,1])$ are fixed, $ \Omega$ is a bounded domain in $ \IR^N$ and $ \lambda \ge 0$ is a varying parameter (see for example Bernstein and Pelesko \cite{BP}).  The function $ u(x)$ denotes the height above a point $ x \in \Omega\subset \IR^N$ of a dielectric membrane clamped on $ \pOm$,  
once it deflects torwards a ground plate fixed at height $z=1$, whenever a positive voltage -- proportional to $\lambda$ -- is applied.

\medskip \noindent In studying this problem,    
one typically makes various simplifying assumptions on the parameters $ \alpha, \beta, \gamma, \chi$, and the  first approximation of (\ref{MEMS.1}) that has been studied extensively so far  is the equation 
\begin{eqnarray*}
\hskip 150pt 
\left\{ \begin{array}{ll}
-\Delta u=   \lambda  \frac{f(x)}{(1-u)^2} &\text{in } \Omega\\
\hfill 0<u<1 \quad \quad &\text{in } \Omega\hskip 150pt (S)_{ \lambda,f} \\
\hfill u=0 \quad \quad \quad &\text{on }\partial \Omega,
\end{array} \right.
\end{eqnarray*}
  where we have set $ \alpha = \beta = \chi=0$ and $ \gamma=1$ (see for example \cite{Esp,EGG,GG} and the monograph \cite{EGG.book}) .   This simple model,  which lends itself to the vast literature on second order semilinear eigenvalue problems, is already a rich source of interesting mathematical problems. The case when the ``permittivity profile" $f$ is constant  ($f=1$) on a general domain  was studied in \cite{MP}, following the pioneering work of Joseph and Lundgren \cite{JL} who had considered the radially symmetric case. The case for a  non constant permittivity profile $ f$ was advocated 
by Pelesko \cite{P}, taken up by \cite{GPW}, and studied in depth in \cite{Esp,EGG,GG}.  The starting point of the analysis is the existence of a pull-in voltage $\lambda^*(\Omega, f)$, defined as
$$ \lambda^*(\Omega, f):= \sup \Big\{ \lambda >0: \hbox{there exists a classical solution of } (S)_{\lambda, f} \Big\}.$$ 
It is then shown  that for every $ 0 < \lambda < \lambda^*$,  there exists a smooth minimal (smallest) solution of $(S)_{\lambda, f}$, while for $ \lambda > \lambda^*$ there is no solution even in a weak sense.  Moreover, the branch $ \lambda \mapsto u_\lambda(x)$ is increasing for each $ x \in \Omega$,  and therefore the function  $u^*(x):= \lim_{\lambda \nearrow \lambda^*} u_\lambda(x)$ can be considered as a generalized solution that corresponds to the pull-in voltage $\lambda^*$.  Now the issue of the regularity of this extremal solution -- which, by elliptic regularity theory, is equivalent to whether $ \sup_\Omega u^*<1$ -- is  an important question for many reasons, not the least of which being the fact that it decides whether the set of solutions stops there, or whether a new branch of solutions emanates from a bifurcation state $(u^*,\lambda^*)$.  This issue turned out to depend closely on the dimension and on the permittivity profile $f$. Indeed, it was shown in \cite{GG} that $u^*$ is regular in dimensions $1\leq N\leq 7$, while it is not necessarily the case for $N\geq 8$. In other words, the dimension $N=7$ is critical for equation $(S)_\lambda$ (when $f=1$, we simplify the notation $(S)_{\lambda,1}$ into $(S)_\lambda$).  On the other hand, it is shown in \cite{EGG} that the regularity of $u^*$ can be restored in any dimension,  provided we allow for a power law profile $|x|^\eta$ with $\eta$ large enough. 

\medskip \noindent The case where $ \beta = \gamma = \chi=0$ (and $ \alpha=1$) in the above model, that is when we are dealing with the following fourth order analog of $(S)_\lambda$ 
\begin{eqnarray*}
\hskip 150pt 
\left\{ \begin{array}{ll}
\bi u=   \frac{\lambda}{(1-u)^2} &\text{in } \Omega\\
0<u<1 &\text{in } \Omega \hskip 150pt (P)_\lambda \\
u=\partial_\nu u=0 &\text{on }\partial \Omega,
\end{array} \right.
\end{eqnarray*}
was also considered by \cite{CDG,LY} but with limited success. One of the reasons is the lack of  a ``maximum principle" which plays such a crucial role in developing the theory for the Laplacian. Indeed, it is a well known fact that 
    such a principle does not normally hold for general domains $ \Omega$ (at least for the clamped boundary conditions $ u = \partial_\nu u =0$ on $ \pOm$) unless one restricts attention to the unit ball $ \Omega =B$   in $ \IR^N$,  where one can exploit  a positivity preserving property of $\bi$ due to T. Boggio \cite{Boggio}.  
  This is precisely what was done in the references mentioned above, where a  theory of the minimal branch associated with $(P)_\lambda$ is developed along the same lines as for $(S)_\lambda$. The second obstacle is the well-known difficulty of extracting energy estimates for solutions of  fourth order problems from their stability properties.  This means that the methods used to analyze the regularity of the extremal solution for $(S)_\lambda$ could not carry to the corresponding problem for $(P)_\lambda$. 
   
\medskip \noindent This is the question we address in this paper as we eventually show the following result.
 
\begin{thm}  The unique extremal solution $u^*$ for $(P)_{\lambda^*}$ in $B$ is regular in dimension $1\leq  N \le 8$, while  it is singular (i.e, $ \sup_B u^*=1$)  for $ N \ge 17$.   
\end{thm}
\noindent Actually, we believe that the critical dimension for $(P)_\lambda$ in $B$ is $N=8$, as opposed to being equal to $7$ in $(S)_\lambda$.  We add that our methods are heavily inspired by the recent paper of Davila et al. \cite{DDGM} where it is shown that $N=12$ is the critical dimension for the fourth order nonlinear eigenvalue problem 
$$\left\{\begin{array}{ll}
\bi u=   \lambda e^u &\text{in } B\\
u=\partial_\nu u=0 &\text{on }\partial B,
\end{array} \right. $$
while the critical dimension for its second order counterpart (i.e.,  the Gelfand problem) is $N=9$.
 
\medskip \noindent Throughout this  paper, we will always consider problem $(P)_\lambda$ on the unit ball $B$. We start by recalling some of the results from \cite{CDG} concerning $(P)_\lambda$, that will be needed in the sequel.  We  define 
$$\lambda^*:= \sup\Big\{ \lambda> 0: \hbox{ there exists a classical solution of }(P)_\lambda \Big\},$$ 
and note that we are not restricting our attention to radial solutions. We will deal also with weak solutions:
\begin{dfn} We say that $u$ is a weak solution of $(P)_\lambda $ if $ 0 \le u \le 1$ a.e. in $B$, $ \frac{1}{(1-u)^2} \in L^1(B)$ and 
\[ \int_B u \Delta^2 \phi = \lambda \int_B \frac{\phi}{(1-u)^2}, \qquad \forall \phi \in C^4(\bar B) \cap H_0^2(B).\] 
We say that $ u $ is a weak super-solution (resp. weak sub-solution) of $(P)_\lambda$ if the equality is replaced with the inequality $ \ge $ (resp. $ \le $) for all $\phi \in C^4(\bar B) \cap H_0^2(B)$ with $\phi \ge 0$.
\end{dfn}  
\noindent We also introduce notions of regularity and stability.
\begin{dfn} Say that a weak solution $u$ of $(P)_\lambda $ is regular (resp. singular) if $\|u\|_\infty<1$ (resp. $=$) and stable (resp. semi-stable) if 
$$\mu_1(u)=\inf \left\{ \int_B ( \Delta \phi)^2 -2 \lambda \int_B \frac{ \phi^2}{(1-u)^3}: \phi \in H_0^2(B), \| \phi \|_{L^2}=1 \right\}$$
is positive (resp. non-negative).
\end{dfn}  
\noindent The following extension of Boggio's principle will be frequently used in the sequel (see \cite[Lemma 16]{AGGM} and \cite[Lemma 2.4]{DDGM}): 
\begin{lemma}[Boggio's Principle]\label{boggio}
Let $u\in L^1(B)$.  Then $u\geq 0$ a.e. in $B$, provided one of the following conditions hold:
\begin{enumerate} 

\item $u\in C^4(\overline{B})$,  $\bi u\geq 0$ on $B$, and  $u=\frac{\partial u}{\partial n}= 0$ on $\partial B$. 
\item $\int_{B} u\bi\phi\,dx\geq 0$ for all $0\leq \phi \in C^4(\overline{B})\cap H_0^2(B)$.
\item $u\in H^2(B)$, $u=0$ and $\frac{\partial u}{\partial n} \leq 0$ on $\partial B$, and  
 $\int_{B} \Delta u \Delta \phi \geq 0$  for all $0\leq \phi \in H^2_0(B)$. 
\end{enumerate} 
Moreover, either $u\equiv 0$ or $u>0$ a.e. in $B$.
\end{lemma}
 
\noindent The following  theorem summarizes the main results in \cite{CDG} that will be needed in the sequel:  
\begin{thm}\label{CdG} The following assertions hold:
\begin{enumerate}
\item For each $ 0 < \lambda < \lambda^*$ there exists a classical minimal solution $u_\lambda$ of $ (P)_\lambda$. Moreover $ u_\lambda $ is radial and radially decreasing.  
\item  For $ \lambda > \lambda^*$,  there are no weak solutions of $(P)_\lambda$.  
\item For each $ x \in B$ the map $ \lambda \mapsto u_\lambda(x)$ is strictly increasing on $ (0,\lambda^*)$.  
\item The pull-in voltage $ \lambda^*$ satisfies the following bounds: 
\[ \max \left\{ \frac{ 32(10N-N^2-12)}{27}, \frac{128-240N+72N^2}{81} \right\} \le \lambda^* \le \frac{4 \nu_1}{27}\] 
where $ \nu_1$ denotes the first eigenvalue of $ \Delta^2 $ in $H_0^2(B)$.  
\item For each $ 0 < \lambda < \lambda^*$, $u_\lambda$ is a stable solution (i.e., $ \mu_1(u_\lambda)>0$). 
\end{enumerate}
\end{thm}  
\noindent Using the stability of $ u_\lambda $,  it can be shown that  $ u_\lambda $ is uniformly bounded in $H_0^2(B)$ and that $ \frac{1}{1-u_\lambda} $ is uniformly bounded in $L^3(B)$.  Since now $ \lambda \mapsto u_\lambda(x)$ is increasing, the function $ u^*(x):= \lim_{ \lambda \nearrow \lambda^*} u_\lambda(x)$ is well defined (in the pointwise sense),  $ u^* \in H_0^2(B)$, $ \frac{1}{1-u^*} \in L^3(B)$ and $ u^*$ is a weak solution of $ (P)_{\lambda^*}$.   Moreover $ u^*$ is the unique weak solution of $(P)_{\lambda^*}$.

\medskip \noindent The second result we list from \cite{CDG} is critical in identifying the extremal solution.  
\begin{thm}  If $ u \in H_0^2(B)$ is a singular weak solution of $(P)_\lambda$, then $ u $ is  semi-stable if and only if $ (u, \lambda) =(u^*,\lambda^*)$. 
\end{thm} 

\section{The effect of boundary conditions on the pull-in voltage}  

As in \cite{DDGM}, we are led to examine problem $(P)_\lambda$ with non-homogeneous boundary conditions such as
$$\hskip 150pt 
\left\{\begin{array}{ll}
\Delta^2 u= \frac{ \lambda}{(1-u)^2} &\hbox{in } B \\
\alpha<u<1 &\hbox{in }B \hskip 150 pt (P)_{\lambda, \alpha, \beta}\\
u= \alpha\:,\:\:\partial_\nu u = \beta &\hbox{on } \partial B, \end{array}\right.
$$
where $\alpha, \beta$ are given. 

\medskip \noindent Notice first that some restrictions on $ \alpha $ and $ \beta$ are necessary. Indeed, letting $\Phi(x):=( \alpha - \frac{\beta}{2} ) + \frac{\beta}{2} |x|^2 $ denote the unique solution of
\begin{equation} \label{Phi} \left\{ \begin{array}{ll}
\Delta^2 \Phi = 0 &\hbox{in } B \\
\Phi = \alpha\:,\:\: \partial_\nu \Phi = \beta&\hbox{on }\partial B,  
\end{array}\right. 
\end{equation} 
we infer immediately from Lemma \ref{boggio} that the function $u-\Phi$ is positive in $B$,  which yields to
$$\sup_B \Phi<\sup_B u\leq 1.$$ 
To insure that $\Phi$ is a classical sub-solution of $(P)_{\lambda,\alpha,\beta}$, we impose $\alpha \not= 1$ and $\beta \leq 0$, and condition $\displaystyle \sup_B \Phi<1$ rewrites as
$\alpha-\frac{\beta}{2} < 1$. We will then say that the pair $ (\alpha, \beta)$ is {\it admissible} if $\beta \leq 0$, and $\alpha-\frac{\beta}{2} < 1$.  

\medskip \noindent This section will be devoted to obtaining  results for $ (P)_{ \lambda, \alpha ,\beta}$ when $ (\alpha, \beta)$ is an admissible pair, which are  analogous to those  for $ (P)_\lambda$.  To cut down on notation, we shall  sometimes drop $ \alpha $ and $ \beta$ from our expressions whenever such an emphasis is not needed.  For example in this section $ u_\lambda $ and $ u^*$ will denote the minimal and extremal solution of $ (P)_{\lambda, \alpha , \beta}$. 

\medskip \noindent We now introduce a notion of weak solution for $(P)_{\lambda,\alpha,\beta}$. 
\begin{dfn} We say that $u$ is a weak solution of $(P)_{\lambda,\alpha,\beta}$ if $\alpha \leq u \le 1$ a.e. in $B$, $ \frac{1}{(1-u)^2} \in L^1(B)$ and if 
\[ \int_B (u-\Phi) \Delta^2 \phi = \lambda \int_B \frac{\phi}{(1-u)^2}, \qquad \forall \phi \in C^4(\bar B) \cap H_0^2(B),\] 
where $\Phi$ is given in (\ref{Phi}). We say that $ u $ is a weak super-solution (resp. weak sub-solution) of $(P)_{\lambda,\alpha,\beta}$ if the equality is replaced with the inequality $ \ge $ (resp. $ \le $) for $ \phi \ge 0$.
\end{dfn}  
\noindent We now define as before 
\[ \lambda^*:= \sup \{ \lambda > 0: (P)_{\lambda, \alpha, \beta} \; \mbox{ has a classical solution} \}\]
and
\[\lambda_*:= \sup \{ \lambda > 0: (P)_{\lambda, \alpha, \beta} \; \mbox{ has a weak solution} \}.\]  
Observe that by the Implicit Function Theorem,  one can always solve $(P)_{\lambda,\alpha,\beta}$ for small $\lambda$'s. Therefore, $\lambda^*$ (and also $\lambda_*$) is well defined.

\medskip \noindent Let now $U$ be a weak super-solution of $(P)_{\lambda,\alpha,\beta}$. Recall the following standard existence result. 
\begin{thm} [\cite{AGGM}] \label{exist} For every $0\leq f \in L^1(B)$,  there exists a unique $0\leq u \in L^1(B)$ which satisfies
$$\int_B u \Delta^2 \phi=\int_B f \phi$$
for all $\phi \in C^4(\bar B) \cap H_0^2(B)$.
\end{thm}  
\noindent We can now introduce the following ``weak iterative scheme": Start with $u_0=U$ and (inductively) let $u_n$, $n \geq 1$, be the solution of
$$\int_B (u_n-\Phi) \Delta^2 \phi=\lambda \int_B  \frac{\phi}{(1-u_{n-1})^2}\qquad\:\forall \: \phi \in C^4(\bar B) \cap H_0^2(B)$$
given by Theorem \ref{exist}. Since $0$ is a sub-solution of $(P)_{\lambda,\alpha,\beta}$, one can easily show inductively by using Lemma \ref{boggio} that $\alpha \leq u_{n+1}\leq u_n \leq U$ for every $n \geq 0$. Since
$$(1-u_n)^{-2}\leq (1-U)^{-2} \in L^1(B),$$
we get by Lebesgue Theorem, that the function $u=\displaystyle \lim_{n \to +\infty} u_n$ is a weak solution of $(P)_{\lambda,\alpha,\beta}$ such that $\alpha \leq u\leq U$. In other words,  the following result holds.  
\begin{thm} \label{super} Assume the existence of a weak super-solution $U$ of $(P)_{\lambda,\alpha,\beta}$. Then there exists a weak solution $u$ of $(P)_{\lambda,\alpha,\beta}$ so that $\alpha \leq u \leq U$ a.e. in $B$. 
\end{thm}  
\noindent In particular,  we can find a weak solution of $(P)_{\lambda,\alpha,\beta}$ for every $\lambda \in (0,\lambda_*)$. Now we show that this is still true for regular weak solutions.

\begin{thm} \label{cch} Let $ (\alpha, \beta)$ be an admissible pair and let $u$ be a weak solution of $(P)_{\lambda,\alpha,\beta}$. Then for every $0<\mu<\lambda$, there is a regular solution for $(P)_{\mu,\alpha,\beta}$.
\end{thm} 
\begin{proof} Let $ \E\in (0,1)$ be given and let  $ \bar u=(1-\E)u+\E \Phi$, where $\Phi$ is given in (\ref{Phi}). We have that
$$\sup_B \bar u\leq (1-\E)+\E \sup_B \Phi<1\:,\quad \inf_B \bar u\geq (1-\E)\alpha +\E \inf_B \Phi=\alpha,$$
and for every $0\leq \phi \in C^4(\bar B) \cap H_0^2(B)$ there holds:
\begin{eqnarray*}
\int_B (\bar u-\Phi) \Delta^2 \phi &=& (1-\E) \int_B (u-\Phi)\Delta^2 \phi
= (1-\E)\lambda \int_B \frac{\phi}{(1-u)^2}\\
&=& (1-\E)^3 \lambda \int_B \frac{\phi}{(1-\bar u+\E (\Phi-1))^2} \geq
(1-\E)^3 \lambda \int_B \frac{\phi}{(1-\bar u)^2}.
\end{eqnarray*} 
Note that $ 0 \le (1-\E)(1-u)=1 - \bar{u}+\E (\Phi -1) <1-\bar u$. So $ \bar{u}$ is a weak super-solution of $ (P)_{ (1-\E)^3 \lambda, \alpha , \beta}$ satisfying $\displaystyle \sup_B \bar u<1$. From Theorem \ref{super} we get the existence of a weak solution $w$ of $ (P)_{ (1-\E)^3 \lambda, \alpha , \beta}$ so that $\alpha \leq w\leq \bar u$. In particular, $\displaystyle \sup_B w<1$ and $w$ is a regular weak solution. Since $\E \in (0,1)$ is arbitrarily chosen, the proof is complete.
\end{proof} 
\noindent Theorem \ref{cch} implies in particular the existence of a regular weak solution $U_\lambda$ for every $\lambda \in (0,\lambda_*)$. Introduce now a ``classical" iterative scheme: $u_0=0$ and (inductively) $u_n=v_n +\Phi$, $n \geq 1$, where $v_n \in H_0^2(B)$ is the (radial) solution of
\begin{equation} \label{pranzo}
\Delta^2 v_n=\Delta^2(u_n-\Phi)= \frac{\lambda}{(1-u_{n-1})^2} \qquad\hbox{in }B.
\end{equation}
Since $v_n \in H_0^2(B)$, $u_n$ is also a weak solution of (\ref{pranzo}), and by Lemma \ref{boggio} we know that $\alpha \leq u_n\leq u_{n+1} \leq U_\lambda$ for every $n \geq 0$. Since $\displaystyle \sup_B u_n \leq \displaystyle \sup_B U_\lambda<1$ for $n\geq 0$, we get that $(1-u_{n-1})^{-2} \in L^2(B)$ and the existence of $v_n$ is guaranteed. Since $v_n$ is easily seen to be uniformly bounded in $H_0^2(B)$, we have that $u_\lambda:=\displaystyle \lim_{n \to +\infty}u_n$ does hold pointwise and weakly in $H^2(B)$. By Lebesgue Theorem, we have that $u_\lambda$ is a radial weak solution of $(P)_{\lambda,\alpha,\beta}$ so that $\displaystyle \sup_B u_\lambda\leq \displaystyle \sup_B U_\lambda<1$. By elliptic regularity theory \cite{ADN} $u_\lambda \in C^\infty(\bar B)$ and $u_\lambda-\Phi=\partial_\nu(u_\lambda-\Phi)=0$ on $\partial B$. So we can integrate by parts to get
$$\int_B \Delta^2 u_\lambda \phi =\int_B \Delta^2(u_\lambda-\Phi) \phi=\int_B (u_\lambda-\Phi)\Delta^2 \phi=\lambda \int_B \frac{\phi}{(1-u_\lambda)^2}$$
for every $\phi \in C^4(\bar B) \cap H_0^2(B)$. Hence, $u_\lambda$ is a radial classical solution of $(P)_{\lambda,\alpha,\beta}$ showing that $\lambda^*=\lambda_*$. Moreover, since $\Phi$ and $v_\lambda:=u_\lambda-\Phi$ are radially decreasing in view of \cite{Sor}, we get that $u_\lambda$ is radially decreasing too. Since the argument above shows that $u_\lambda<U$ for any other classical solution $U$ of $(P)_{\mu,\alpha,\beta}$ with $\mu \geq \lambda$, we have that $u_\lambda$ is exactly the minimal solution and $u_\lambda$ is strictly increasing as $\lambda \uparrow \lambda^*$. In particular, we can define $ u^*$ in the usual way: $ u^*(x)= \displaystyle \lim_{\lambda \nearrow \lambda^*} u_\lambda(x)$. 

\medskip \noindent Finally, we show the finiteness of the pull-in voltage. 
\begin{thm} If $ (\alpha, \beta)$ is an admissible pair, then $\lambda^*(\alpha, \beta) <+\infty$. 
\end{thm}
\begin{proof} Let $u$ be a classical solution of $ (P)_{\lambda, \alpha, \beta}$  and let $ (\psi, \nu_1)$ denote the first eigenpair of $ \Delta^2$ in $H_0^2(B)$ with $ \psi >0$.  Now, let $ C $ be such that 
\[ \int_{\partial B} (\beta \Delta \psi - \alpha \partial_\nu \Delta \psi) = C \int_B \psi. \] 
Multiplying $ (P)_{\lambda,\alpha,\beta}$ by $ \psi$ and then integrating by parts one arrives at 
\[ \int_B \left( \frac{ \lambda}{(1-u)^2} - \nu_1 u -C \right) \psi =0. \] 
Since $ \psi>0$ there must exist a point $\bar x \in B$ where  
$\frac{ \lambda}{(1-u(\bar x))^2} - \nu_1 u(\bar x) -C \le 0.$
Since $\alpha<u(\bar x)<1$, one can conclude that 
$ \lambda \le \sup_{\alpha< u <1} ( \nu_1 u +C)(1-u)^2$,  
which shows that $ \lambda^*<+\infty$.
\end{proof}  
\noindent The following summarizes what we have shown so far.

\begin{thm} If $(\alpha,\beta)$ is an admissible pair, then $\lambda^* \in (0,+\infty)$ and the following hold:
\begin{enumerate}
\item For each $ 0 < \lambda < \lambda^*$ there exists a classical, minimal solution $u_\lambda$ of $(P)_{\lambda,\alpha,\beta}$.  Moreover $ u_\lambda $ is radial and radially decreasing.
\item For each $ x \in B$ the map $ \lambda \mapsto u_\lambda(x)$ is strictly increasing on $ (0,\lambda^*)$.
\item For $ \lambda > \lambda^*$ there are no weak solutions of $(P)_{\lambda,\alpha,\beta}$.
\end{enumerate}
\label{quasi} \end{thm}

\subsection{Stability of the minimal branch of solutions} 
This section is devoted to the proof of the following stability result for minimal solutions.  We shall need yet another notion of $H^2(B)-$weak solutions, which is an intermediate class between classical and weak solutions.

\begin{dfn} We say that $u$ is a $H^2(B)-$weak solution of $(P)_{\lambda,\alpha,\beta}$ if $u -\Phi \in H_0^2(B)$, $ \alpha \le u \le 1$ a.e. in $B$, $ \frac{1}{(1-u)^2} \in L^1(B)$ and if
\[ \int_B \Delta u \Delta \phi = \lambda \int_B \frac{\phi}{(1-u)^2}, \qquad \forall \phi \in C^4(\bar B) \cap H_0^2(B),\] 
where $\Phi$ is given in (\ref{Phi}). We say that $u$ is a $H^2(B)-$weak super-solution (resp. $H^2(B)-$weak sub-solution) of $(P)_{\lambda,\alpha,\beta}$ if for $ \phi \ge 0$ the equality is replaced with $ \ge$ (resp. $ \le $) and $u\geq \alpha$ (resp. $\leq$), $\partial_\nu u \leq \beta$ (resp. $\geq$) on $\partial B$.
\end{dfn}  

\begin{thm} \label{stable} Suppose $ (\alpha,\beta)$ is an admissible pair.  
\begin{enumerate}
\item  The minimal solution $ u_\lambda $ is then stable and is the unique semi-stable $H^2(B)-$weak solution of $(P)_{\lambda,\alpha,\beta}$. 
\item The function $ u^*:= \displaystyle \lim_{\lambda \nearrow \lambda^*} u_\lambda$ is a well-defined semi-stable $H^2(B)-$weak solution of  $(P)_{\lambda^*,\alpha,  \beta}$.
\item When $u^*$ is classical solution, then $\mu_1(u^*)=0$ and $u^*$ is the unique $H^2(B)-$weak solution of $(P)_{\lambda^*,\alpha,\beta}$.
\item If $ v$ is a singular, semi-stable $H^2(B)-$weak solution of $ (P)_{ \lambda, \alpha, \beta}$, then $ v=u^*$ and $ \lambda = \lambda^*$
\end{enumerate}
\end{thm}

\noindent The crucial tool is a comparison result which is valid exactly in this class of solutions.
\begin{lemma} \label{shi} Let $ (\alpha, \beta)$ be an admissible pair and $u$ be a semi-stable $H^2(B)-$weak solution of $(P)_{\lambda, \alpha, \beta}$. Assume $ U $ is a $H^2(B)-$weak super-solution of $(P)_{\lambda, \alpha, \beta}$ so that $U-\Phi \in H_0^2(B)$. Then 
\begin{enumerate}
\item $ u \le U$ a.e. in $B$;
\item If $ u$ is a classical solution and $ \mu_1(u)=0$ then $ U=u$.
\end{enumerate}
\end{lemma} 
\begin{proof}  (i)  Define $ w:= u-U$.   Then by the Moreau decomposition \cite{M} for the biharmonic operator,  there exist $ w_1,w_2 \in H_0^2(B)$, with $ w=w_1 + w_2$, $ w_1 \ge 0$ a.e., $\Delta^2 w_2 \le 0 $ in the $H^2(B)-$weak sense and $\int_B \Delta w_1 \Delta w_2=0$.  By Lemma \ref{boggio}, we have that $w_2 \le 0$ a.e. in $B$.\\  
Given now $ 0 \le \phi \in C_c^\infty(B)$, we have that
\[ \int_B \Delta w \Delta \phi \leq  \lambda \int_B (f(u) - f(U)) \phi, \] 
where $ f(u):= (1-u)^{-2}$.  Since $ u$ is semi-stable, one has 
\begin{eqnarray*} 
\lambda \int_B f'(u) w_1^2 \le  \int_B (\Delta w_1)^2 
= \int_B \Delta w \Delta w_1 \le  \lambda \int_B ( f(u) - f(U)) w_1.
\end{eqnarray*} 
Since $ w_1 \ge w$ one also has 
\[ \int_B f'(u) w w_1 \le \int_B (f(u)-f(U)) w_1,\]  
which once re-arranged gives 
\[ \int_B \tilde{f} w_1 \geq 0,\] 
where $ \tilde{f}(u)= f(u) - f(U) -f'(u)(u-U)$. The strict convexity of $f$ gives $ \tilde{f} \le 0$ and $ \tilde{f}< 0 $ whenever $u \not= U$. Since $w_1 \ge 0$ a.e. in $B$ one sees that $ w \le 0 $ a.e.  in $B$. The inequality $ u \le U$  a.e. in $B$ is then established. 

\medskip \noindent (ii)  Since $u$ is a classical solution, it is easy to see that the infimum in $\mu_1(u)$ is attained at some $\phi$. The function $\phi$ is then the first eigenfunction of $\Delta^2-\frac{2\lambda}{(1-u)^3}$ in $H_0^2(B)$. Now we show that $ \phi$ is of fixed sign.  Using the above decomposition, one has $ \phi= \phi_1 + \phi_2$ where $ \phi_i \in H_0^2(B)$ for $i=1,2$, $ \phi_1 \ge 0$, $ \int_B \Delta \phi_1 \Delta \phi_2=0$ and $ \Delta^2 \phi_2 \le 0$ in the $H^2_0(B)-$weak sense. If $ \phi$ changes sign,  then $ \phi_1 \not\equiv 0$ and $ \phi_2 <0$ in $B$ (recall that either $\phi_2<0$ or $\phi_2=0$ a.e. in $B$). We can write now:
\begin{eqnarray*}
0 = \mu_1(u) 
\le  \frac{ \int_B (\Delta (\phi_1 -\phi_2))^2 - \lambda f'(u) ( \phi_1 - \phi_2)^2}{ \int_B ( \phi_1 - \phi_2)^2} < \frac{ \int_B ( \Delta \phi)^2 - \lambda f'(u) \phi^2 }{ \int_B \phi^2} =\mu_1(u)
\end{eqnarray*} 
in view of $\phi_1 \phi_2<-\phi_1\phi_2$ in a set of positive measure, leading to a contradiction.\\  
So we can assume $ \phi \ge 0$, and by the Boggi's principle we have $\phi>0$ in $B$. For $ 0 \le t \le 1$ define 
$$g(t)=\int_B \Delta \left[t U+(1-t)u \right] \Delta \phi
- \lambda \int_B f( tU+(1-t)u) \phi,$$  
where $\phi$ is the above first eigenfunction.
Since $ f$ is convex one sees that 
$$g(t)\geq \lambda \int_B \left[t f(U)+(1-t)f(u)-f(tU+(1-t)u)\right]\phi \geq 0$$  for every $t \geq 0$. Since $ g(0) =0$ and 
$$ g'(0)= \int_B \Delta (U-u) \Delta \phi-\lambda f'(u)(U-u)\phi=0 ,$$  
we get that
\[ g''(0)=- \lambda \int_B f''(u) (U-u)^2 \phi\geq 0.\] 
Since $f''(u)\phi>0$ in $B$, we finally get that $ U=u$ a.e. in $B$. 
\end{proof} 

\noindent Based again on Lemma \ref{boggio}(3), we can show a more general version of the above Lemma \ref{shi}.
\begin{lemma} \label{poo} Let $ (\alpha,\beta)$ be an admissible pair and $\beta'\leq 0$. Let $u$ be a semi-stable $H^2(B)-$weak sub-solution of $(P)_{\lambda, \alpha,\beta}$ with $u=\alpha$, $\partial_\nu u=\beta' \geq \beta$ on $\partial B$. Assume that $U$ is a $H^2(B)-$weak super-solution of $(P)_{\lambda, \alpha ,\beta}$ with $U=\alpha$, $\partial_\nu U=\beta$ on $\partial B$. Then $ U \ge u$ a.e. in $B$. 
\end{lemma} 
\begin{proof}  Let $ \tilde{u} \in H_0^2(B)$ denote a weak solution to $ \Delta^2 \tilde{u}= \Delta^2 (u-U)$ in $B$. Since $\tilde{u}-u+U=0$ and $\partial_\nu(\tilde{u}-u+U)\leq 0$ on $\partial B$, by Lemma \ref{boggio} one has that $ \tilde{u} \ge u-U $ a.e. in $B$.  Again by the Moreau decomposition \cite{M},  we may write $\tilde u$ as $ \tilde{u} = w+v $, where $ w,v \in H_0^2(B)$, $ w \ge 0 $ a.e. in $B$, $ \Delta^2 v \le 0$ in a $H^2(B)-$weak sense and $\int_B \Delta w \Delta v=0$.   Then for $ 0 \le \phi \in C^4 (\bar B)\cap H_0^2(B)$ one has 
\[ \int_B \Delta \tilde{u} \Delta \phi =\int_B \Delta(u-U) \Delta \phi \leq \lambda \int_B (f(u)- f( U)) \phi .\] 
In particular, we have that
\[ \int_B \Delta \tilde{u} \Delta w \le \lambda \int_B ( f(u)-f(U)) w.\]  
Since by semi-stability of $u$ 
\begin{eqnarray*}
\lambda \int_B f'(u) w^2\leq \int_B ( \Delta w)^2 = \int_B \Delta \tilde{u} \Delta w ,
\end{eqnarray*} 
we get that 
\[  \int_B f'(u) w^2 \le \int_B ( f(u)-f(U)) w.\]  By Lemma \ref{boggio} we have $v\leq 0$ and then $ w \ge \tilde{u} \ge u -U$ a.e. in $B$. So we see that 
\[ 0 \le \int_B \left( f(u)-f(U)-f'(u)(u-U) \right) w.\]  The strict convexity of $ f$ implies as in Lemma \ref{shi} that $ U \ge u $ a.e. in $B$. 
\end{proof}
\noindent We shall need the following a-priori estimates along the minimal branch $u_\lambda$.

\begin{lemma} \label{extremalsol} Let $ (\alpha, \beta)$ be an admissible pair. Then one has  \[ 2 \int_B \frac{( u_\lambda - \Phi)^2}{(1-u_\lambda)^3} \le \int_B \frac{ u_\lambda - \Phi}{(1-u_\lambda)^2},\] 
where $ \Phi$ is given in (\ref{Phi}). In particular, there is a constant $C>0$ so that for every $\lambda \in (0,\lambda^*)$, we have
\begin{equation} \label{tardi}
\int_B (\Delta u_\lambda)^2+\int_B \frac{1}{(1-u_\lambda)^3} \leq C.
\end{equation} 

\end{lemma}  
\begin{proof} Testing $ (P)_{\lambda, \alpha , \beta}$ on $ u_\lambda - \Phi \in C^4(\bar B) \cap H^2_0(B)$, we see that 
\begin{eqnarray*}
\lambda \int_B \frac{ u_\lambda - \Phi}{(1-u_\lambda)^2} = \int_B \Delta u_\lambda \Delta( u_\lambda - \Phi) =\int_B ( \Delta (u_\lambda - \Phi))^2 
\ge 2 \lambda \int_B \frac{ (u_\lambda- \Phi)^2}{( 1-u_\lambda)^3}
\end{eqnarray*}
in view of $\Delta^2 \Phi=0$. In particular, for $\delta>0$ small we have that
\begin{eqnarray*}
\int_{\{|u_\lambda-\Phi| \geq \delta \}}\frac{1}{(1-u_\lambda)^3}&\leq &
\frac{1}{\delta^2} \int_{\{|u_\lambda-\Phi| \geq \delta \}}\frac{(u_\lambda-\Phi)^2}{(1-u_\lambda)^3} 
\leq \frac{1}{\delta^2} \int_B \frac{1}{(1-u_\lambda)^2}\\
&\leq &\delta \int_{\{|u_\lambda-\Phi| \geq \delta \}}\frac{1}{(1-u_\lambda)^3}+C_\delta
\end{eqnarray*}
by means of Young's inequality.
Since for $\delta$ small, 
$$\int_{\{|u_\lambda-\Phi| \leq \delta \}}\frac{1}{(1-u_\lambda)^3}\leq C'$$
for some $C'>0$, we  can deduce that for every $\lambda \in (0,\lambda^*)$, 
$$\int_B \frac{1}{(1-u_\lambda)^3} \leq C$$
for some $C>0$.  By Young's and H\"older's inequalities, we now have
$$\int_B (\Delta u_\lambda)^2=\int_B \Delta u_\lambda \Delta \Phi+\lambda \int_B \frac{u_\lambda -\Phi}{(1-u_\lambda)^2}\leq \delta \int_B (\Delta u_\lambda)^2 
+C_\delta+C \left(\int_B \frac{1}{(1-u_\lambda)^3} \right)^{\frac{2}{3}}$$
 and estimate (\ref{tardi}) is therefore established.\end{proof} 

\medskip \noindent We are now ready to establish Theorem \ref{stable}.\\
{\bf Proof (of Theorem \ref{stable}):} (1)\, Since $\|u_\lambda\|_\infty <1$, the infimum defining $\mu_1(u_\lambda)$ is achieved at a first eigenfunction for every $\lambda \in (0,\lambda^*)$. Since $\lambda  \mapsto u_\lambda(x)$ is increasing for every $x \in B$, it is easily seen that $\lambda \mapsto \mu_1( u_\lambda)$ is an increasing, continuous function on $ (0, \lambda^*)$.  Define 
\[ \lambda_{**}:= \sup\{ 0 <\lambda < \lambda^*: \: \mu_1( u_\lambda) >0 \} .\] We have that $ \lambda_{**}= \lambda^*$. Indeed, otherwise we would have that $ \mu_1(u_{ \lambda_{**}}) =0$, and for every $ \mu \in ( \lambda_{**}, \lambda^*)$ $ u_{\mu}$ would be a classical super-solution of $ (P)_{ \lambda_{**},\alpha, \beta}$. A contradiction arises since Lemma \ref{shi} implies $u_{\mu} = u_{\lambda_{**}}$.\\  
Finally, Lemma \ref{shi} guarantees uniqueness in the class of semi-stable $H^2(B)-$weak solutions.\\
(2) \, By estimate (\ref{tardi}) it follows that $u_\lambda \to u^*$ in a pointwise sense and weakly in $H^2(B)$, and $ \frac{1}{1-u^*} \in L^3(B)$. In particular, $u^*$ is a $H^2(B)-$weak solution of $(P)_{ \lambda^*, \alpha ,\beta}$ which is also semi-stable as limiting function of the semi-stable solutions $\{u_\lambda\}$.\\  
(3) Whenever $\|u^*\|_\infty<1$, the function $u^*$ is a classical solution, and by the Implicit Function Theorem  we have that $\mu_1(u^*)=0$ to prevent the continuation of the minimal branch beyond $\lambda^*$. By Lemma \ref{shi} $u^*$ is then the unique $H^2(B)-$weak solution of $(P)_{\lambda^*,\alpha,\beta}$. An alternative approach --which we do not pursue here-- based on the very definition of the extremal solution $u^*$ is available in \cite{CDG} when $\alpha=\beta=0$ (see also \cite{Mar}) to show that $u^*$ is the unique weak solution of $(P)_{\lambda^*}$, regardless of whether $u^*$ is regular or not.\\
(4) \, If $ \lambda < \lambda^*$, by uniqueness $v=u_\lambda $. So $v$ is not singular and a contradiction arises. 

\medskip \noindent By Theorem \ref{quasi}(3) we have that $ \lambda = \lambda^*$. Since $ v $ is a semi-stable $H^2(B)-$weak solution of $ (P)_{ \lambda^*, \alpha, \beta}$ and $ u^*$ is a $H^2(B)-$weak super-solution of $ (P)_{\lambda^*, \alpha , \beta}$, we can apply Lemma \ref{shi} to get $ v \le u^*$ a.e. in $B$.  Since $u^*$ is a semi-stable solution too, we can reverse the roles of $ v$ and $ u^*$ in Lemma \ref{shi} to see that $ v \ge u^*$ a.e. in $B$. So equality $v=u^*$ holds and the proof is done.

\section{Regularity of the extremal solution  for $ 1 \le N \le 8$ }  
We now return to the issue of the regularity of the extremal solution in problem $(P)_\lambda$. Unless stated otherwise, $ u_\lambda $ and $ u^*$ refer to the minimal and extremal solutions of $ (P)_\lambda$.   We shall show that the extremal solution $ u^*$ is regular provided $ 1 \le N \le 8$. 
We first begin by showing that it is indeed the case  in small dimensions:   
\begin{thm} $ u^*$ is regular in dimensions $ 1 \le N \le 4$. 
\label{regular1} \end{thm} 
\begin{proof} As already observed, estimate (\ref{tardi}) implies that $f(u^*)=(1-u^*)^{-2} \in L^{\frac{3}{2}}(B)$. Since $u^*$ is radial and radially decreasing, we need to show that $ u^*(0)<1$ to get the regularity of $ u^*$. The integrability of $f(u^*)$ along with elliptic regularity theory shows that $ u^* \in W^{4, \frac{3}{2}}(B)$. By the Sobolev imbedding Theorem we get that $u^*$ is a Lipschitz function in $B$.\\
Now suppose $ u^*(0)=1$ and $ 1 \le N \le 3$. Since 
$$\frac{1}{1-u} \ge \frac{C}{|x|}\qquad \hbox{in }B$$ 
for some $ C>0$, one sees that 
\[ \infty = C^3 \int_B \frac{1}{|x|^3} \le \int_B \frac{1}{(1-u^*)^3} < \infty.\] 
A contradiction arises and hence $u^*$ is regular for $ 1 \le N \le 3$.\\
For $N=4$ we need to be more careful and observe that $u^* \in C^{1, \frac{1}{3}}(\bar B)$ by the Sobolev Imbedding Theorem. If $ u^*(0)=1$, then $ \nabla u^*(0)=0$ and 
\[ \frac{1}{1-u^*} \ge \frac{C}{|x|^\frac{4}{3}} \qquad \hbox{in }B \]
for some $ C>0$. We now obtain a contradiction exactly as above. 
\end{proof}   
\noindent We now tackle the regularity of $ u^*$ for $ 5 \le N \le 8$. We start with the following crucial result:
\begin{thm} Let $ N \ge 5$ and $ (u^*, \lambda^*)$ be the extremal pair of $(P)_\lambda$. When $u^*$ is singular, then 
\[ 1-u^*(x) \le C_0 |x|^\frac{4}{3} \qquad \hbox{in }B,\] 
where $ C_0:= \left( \frac{\lambda^*}{\overline{\lambda}}\right)^\frac{1}{3}$ and $ \bar{\lambda}:= \frac{8 (N-\frac{2}{3}) (N- \frac{8}{3})}{9}$.
\label{touchdown}\end{thm}  
\begin{proof} First note that Theorem \ref{CdG}(4) gives the lower bound:
\begin{equation}\label{lowbound}
\lambda^* \geq \bar \lambda=  \frac{128-240N+72N^2}{81}.
\end{equation}
For $ \delta >0$, we define $ u_\delta(x):=1-C_\delta |x|^\frac{4}{3}$ with $ C_\delta:= \left( \frac{\lambda^*}{\bar \lambda}+\delta \right)^\frac{1}{3}>1$.   Since $N\geq 5$, we have that $ u_\delta \in H^2_{loc}(\IR^N)$, $ \frac{1}{1-u_\delta} \in L^3_{loc}(\IR^N)$ and $ u_\delta $ is a $H^2-$weak solution of
\[ \Delta^2 u_\delta = \frac{ \lambda^* + \delta \bar{ \lambda}}{ (1-u_\delta)^2} \qquad \mbox{ in } \IR^N.\]  
We claim that $u_\delta \leq u^*$ in $B$, which will finish  the proof  by just letting $\delta \to 0$.

\medskip \noindent Assume by contradiction that the set $\Gamma:=\{ r \in (0,1):u_\delta(r) >u^*(r) \}$ is non-empty, and let $r_1=\displaystyle \sup \:\Gamma$.
Since 
\[ u_\delta(1) = 1 - C_\delta<0=u^*(1),\] 
we have that $0 < r_1 < 1$ and one infers that 
\[ \alpha:= u^*(r_1)=u_\delta(r_1) \:, \quad \beta:=( u^*)'(r_1) \geq u_\delta'(r_1) .\]    
Setting $u_{\delta,r_1}(r)=r_1^{-\frac{4}{3}}\left(u_\delta(r_1 r)-1 \right)  +1$, we easily see that $u_{\delta,r_1}$ is a $H^2(B)-$weak super-solution of $(P)_{\lambda^*+\delta \bar \lambda,\alpha',\beta'}$, where
$$\alpha':= r_1^{-\frac{4}{3}}( \alpha-1) +1\:,\quad \beta':=
r_1^{-\frac{1}{3}} \beta.$$

\medskip \noindent Similarly, let us define $u^*_{r_1}(r)= r_1^{-\frac{4}{3}}\left( u^*(r_1 r)-1\right) +1$. The dilation map 
\begin{equation}
w \to w_{r_1}(r)=r_1^{-\frac{4}{3}}\left( w(r_1 r)-1\right) +1
\end{equation}
 is a correspondence between solutions of $(P)_{\lambda}$ on $B$ and of $(P)_{\lambda,1-r_1^{-\frac{4}{3}},0}$ on $B_{r_1^{-1}}$ which preserves the $H^2-$integrability. In particular, $(u^*_{r_1},\lambda^*)$ is  the extremal pair of $(P)_{\lambda,1-r_1^{-\frac{4}{3}},0}$ on $B_{r_1^{-1}}$ (defined in the obvious way). 
Moreover, $u^*_{r_1}$ is a singular semi-stable $H^2(B)-$ weak solution of $(P)_{\lambda^* ,\alpha',\beta'}$.

\medskip \noindent Since $u^*$ is radially decreasing, we have that $ \beta' \le 0$. Define the function $w$ as $w(x):= ( \alpha' - \frac{\beta'}{2}) + \frac{ \beta'}{2} |x|^2 + \gamma(x) $, where $ \gamma $ is a solution of $ \Delta^2 \gamma= \lambda^* $ in $B$ with $ \gamma = \partial_\nu \gamma =0 $ on $ \partial B$.  Then $ w$ is a classical solution of 
$$ \left\{ \begin{array}{ll}
\Delta^2 w = \lambda^* &\hbox{in } B \\
w = \alpha'\:, \quad \partial_\nu w = \beta' &\hbox{on } \partial B.
\end{array} \right.$$ 
Since $\frac{\lambda^*}{(1-u^*_{r_1})^2}\geq \lambda^*$, by Lemma \ref{boggio} we have $ u^*_{r_1} \ge w $ a.e. in $B$.  Since $ w(0) = \alpha' - \frac{\beta'}{2}+ \gamma(0)$ and $ \gamma(0)>0$, the bound $ u^*_{r_1} \le 1$ a.e. in $B$ yields to $ \alpha' - \frac{\beta'}{2}<1$. Namely, $ (\alpha', \beta')$ is an admissible pair and by Theorem 
\ref{stable}(4) we get that $(u^*_{r_1},\lambda^*)$ coincides with the extremal pair of $(P)_{ \lambda, \alpha', \beta'}$ in $B$.

\medskip \noindent Since $(\alpha',\beta')$ is an admissible pair and $u_{\delta,r_1}$ is a $H^2(B)-$weak super-solution of $(P)_{\lambda^*+\delta \bar \lambda,\alpha',\beta'}$, by Theorem \ref{super} we get the existence of a weak solution of $(P)_{\lambda^*+\delta \bar \lambda,\alpha',\beta'}$. Since $\lambda^*+\delta \bar \lambda>\lambda^*$, we contradict the fact that $\lambda^*$ is the extremal parameter of $(P)_{\lambda,\alpha',\beta'}$.
\end{proof} 
\noindent Thanks to this lower estimate on $u^*$, we get the following result. 
\begin{thm} If $ 5 \le N \le 8$, then the extremal solution $u^*$ of $ (P)_\lambda$ is regular. 
\label{regular2} \end{thm} 
\begin{proof} Assume that $ u^*$ is singular. For  $ \E>0$ set  $\psi(x):= |x|^{ \frac{4-N}{2}+\E}$ and note that 
\[ (\Delta \psi)^2 = (H_N +O( \E)) |x|^{-N+2\E}, \qquad \mbox{ where}\qquad H_N:= \frac{N^2 (N-4)^2}{16}.\]  
Given $\eta \in C_0^\infty(B)$, and since $N\geq 5$, we can use the test function $\eta \psi \in H_0^2(B)$ into the stability inequality to obtain 
\[ 
2 \lambda \int_B \frac{\psi^2}{(1-u^*)^3} \le \int_B (\Delta \psi)^2 +O(1),
\] 
where $O(1)$ is a bounded function as $ \E \searrow 0$. By Theorem \ref{touchdown} we find that
\[   2 \bar \lambda \int_B \frac{\psi^2}{|x|^4}  \le \int_B (\Delta \psi)^2 +O(1),\] 
and then
\[   2 \bar \lambda \int_B |x|^{-N+2\E} \le (H_N +O(\E)) \int_B |x|^{-N+2\E} +O(1).\] Computing the integrals one arrives at 
\[ 2 \bar \lambda \le H_N +O(\E).\] 
As $ \E \to 0$ finally we obtain $2 \bar \lambda \le H_N$. Graphing this relation one sees that $ N \ge 9$.
\end{proof}    
\noindent We can now slightly improve the lower bound (\ref{lowbound}).
\begin{cor} \label{lambda.bar}  In any dimension $N\geq 1$, we have 
\begin{equation}\label{lower}
\lambda^*>\bar \lambda=\frac{8 (N-\frac{2}{3}) (N- \frac{8}{3})}{9}.
\end{equation}
\end{cor}
\begin{proof} The function $\bar{u}:=1-|x|^\frac{4}{3}$ is a $H^2(B)-$ weak solution of $(P)_{\bar \lambda,0,-\frac{4}{3}}$. If by contradiction $\lambda^*=\bar \lambda$, then $\bar u$ is a $H^2(B)-$weak super-solution of $(P)_\lambda$ for every $\lambda \in (0,\lambda^*)$. By Lemma \ref{shi} we get that $u_\lambda \le \bar{u}$ for all $ \lambda < \lambda^*$, and then $u^*\le \bar{u}$ a.e. in $B$. 

\medskip \noindent If $1\leq N \leq 8$, $u^*$ is then regular by Theorems \ref{regular1} and \ref{regular2}. By Theorem \ref{stable}(3) there holds $\mu_1(u^*)=0$. Lemma \ref{shi} then yields that $u^*=\bar u$, which is a contradiction since then $u^*$ will not satisfy the boundary conditions. 

\medskip \noindent If now $N\geq 9$ and $ \bar{\lambda} = \lambda^*$, then $C_0=1$ in Theorem \ref{touchdown}, and we then have  $ u^* \geq \bar{u}$. It means again that $u^*=\bar u$, a contradiction that completes the proof. 
\end{proof}

\section{The extremal solution is singular for $N \ge 17$}   

In this section,  we will need the following improved Hardy-Rellich inequality, which is valid for $N \ge 5$ (see \cite{GM} and references therein): 
\[ \int_B (\Delta \psi)^2 \ge H_N \int_B \frac{\psi^2}{|x|^4} + C \int_B \psi^2 \qquad \forall \; \psi \in H_0^2(B) \] where $ H_N:= \frac{ N^2(N-4)^2}{16} $ is optimal and $C>0$.  As in the previous section $(u^*,\lambda^*)$ denotes the extremal pair of $ (P)_\lambda$. We first show the following upper bound on $u^*$.
\begin{lemma} \label{blow} If $ N \ge 9$, then $ u^* \le 1 - |x|^\frac{4}{3}$ in $B$. 
\end{lemma}
\begin{proof} Recall that $ \bar{\lambda}:=\frac{8 (N-\frac{2}{3}) (N- \frac{8}{3})}{9} \leq \lambda^*$. If $ \bar{\lambda} = \lambda^*$, then by the proof of Corollary \ref{lambda.bar}, we know that $u^*\leq \bar u$.\\
Suppose now that $ \bar{\lambda} < \lambda^*$.   We claim that $ u_\lambda \le \bar{u}$ for all $ \lambda \in ( \bar{\lambda}, \lambda^*)$.  Indeed,  fix $ \lambda $ and assume by contradiction that  
\[ R_1:= \inf \{ 0 \le R \le 1: u_\lambda < \bar{u} \mbox{ in } (R,1) \}>0.\]   
From the boundary conditions, one has that $ u_\lambda(r) < \bar{u}(r)$ as $ r\to 1^-$. Hence, $0<R_1<1$, $ \alpha:=u_\lambda(R_1)=\bar{u}(R_1)$ and $ \beta:=u_\lambda'(R_1) \le \bar{u}'(R_1)$.  Introduce,  as in the proof of Theorem \ref{touchdown}, the functions  $(u_\lambda)_{R_1}$ and $(\bar u)_{R_1}$.  We have that $(u_\lambda)_{R_1}$ is a classical super-solution of $(P)_{\bar \lambda,\alpha',\beta'}$, where
$$\alpha':= R_1^{-\frac{4}{3}}( \alpha-1) +1\:,\quad \beta':=
R_1^{-\frac{1}{3}} \beta.$$
Note that $(\bar u)_{R_1}$ is a $H^2(B)-$weak sub-solution of $(P)_{\bar \lambda,\alpha',\beta'}$ which is also semi-stable in view of $2\bar \lambda \leq H_N$ and the Hardy inequality. By Lemma \ref{poo},  we deduce that $(u_\lambda)_{R_1}\geq (\bar u)_{R_1}$ in $B$. Note that, arguing as in the proof of Theorem \ref{touchdown}, $(\alpha',\beta')$ is an admissible pair.

\medskip \noindent We have therefore shown that $u_\lambda \geq \bar u$ in $B_{R_1}$ and a contradiction arises in view of the fact that $\displaystyle \lim_{x \to 0} \bar u(x)=1$ and $\|u_\lambda\|_\infty<1$. It follows that  $u_\lambda \leq \bar u$ in $B$ for every $\lambda \in (\bar \lambda, \lambda^*)$, and in particular $u^*\leq \bar u$ in $B$.
\end{proof}   
\noindent Our approach for showing that  $ u^*$ is singular for large dimensions, will depend on the sign of $ H_N - 2 \lambda^*$.

\begin{thm} If $ N \ge 9$ and $\lambda^* \le \frac{H_N}{2}$, then the extremal solution $ u^*$ of $(P)_\lambda$ is singular. \label{kkl}
\end{thm} 
\begin{proof} Let $ \psi \in C_c^\infty(B)$ with $ \int_B \psi^2 =1$.  By Lemma \ref{blow} and the improved Hardy-Rellich inequality (see \cite{GM}), one then has 
\begin{eqnarray*}
\int_B (\Delta \psi)^2 - 2 \lambda^* \int_B \frac{ \psi^2}{(1-u^*)^3} \ge \int_B (\Delta \psi)^2 - H_N \int_B \frac{\psi^2}{|x|^4} \ge C.
\end{eqnarray*} 
It follows that $ \mu_1(u^*)>0$ and $ u^*$ must be singular, since otherwise, one could use the Implicit Function Theorem  to continue the minimal branch beyond $ \lambda^*$. 
\end{proof}  
\noindent We can now show the following result about the extremal solution.
\begin{thm} \label{hjh} The following upper bounds on $\lambda^*$ hold in large dimensions.
\begin{enumerate}
\item  If $N \ge 31$, then $\lambda^* \le 27 \bar{ \lambda}\leq \frac{H_N}{2}$.
\item If $17 \le N \le 30$, then $\lambda^* \le \frac{H_N}{2}$.
\end{enumerate}
The extremal solution is therefore singular for dimension $N\geq 17$.
\end{thm} 
\begin{proof} Consider for any $m>0$ the following function:
\begin{equation}
 w_m:=1-3m/(3m-4)r^{4/3}+4 r^m/(3m-4) 
\end{equation}
Assume first that  $N \ge 31$, then $27 \bar{ \lambda} \le \frac{H_N}{2}$. We shall show that $w_2$ is a singular  $H^2(B)-$weak sub-solution of $(P)_{27 \bar \lambda}  $ that is semi-stable.  Indeed,  write 
\[
w_2:=1-|x|^{\frac{4}{3}}-2(|x|^\frac{4}{3}-|x|^2)=\bar{u}-\phi_0,
\]
 where $ \phi_0:=2(|x|^\frac{4}{3}-|x|^2)$, and note that  $w_2\in H_0^2(B)$, $ \frac{1}{1-w_2} \in L^3(B)$, $ 0 \le w_2 \le 1$ in $B$, and 
\[ 
\Delta^2 w_2 \le \frac{ 27 \bar{\lambda}}{(1-w_2)^2} \qquad \hbox{in }B\setminus \{0\}. 
\] 
So $w_2$ is $H^2(B)-$weak sub-solution of $(P)_{27 \bar \lambda} $. Moreover, since  $ 27 \bar{ \lambda} \le \frac{H_N}{2}$, and since $\phi_0\geq 0$, we get that 
\begin{eqnarray*}
54 \bar \lambda \int_B \frac{ \psi^2}{(1-w_2)^3} \le H_N \int_B \frac{ \psi^2}{(|x|^\frac{4}{3}+\phi_0)^3} \le H_N \int_B \frac{\psi^2}{|x|^4} \le \int_B (\Delta \psi)^2
\end{eqnarray*} 
for all $ \psi \in H_0^2(B)$. Hence $w_2$ is also semi-stable. If now $ 27 \bar{ \lambda} <  \lambda^*$, then by Lemma \ref{poo},  $ w_2$ is necessarily below the minimal solution $ u_{ 27 \bar{ \lambda}}$ which contradicts the fact that $ w_2 $ is singular.  Hence $ \lambda^* \leq 27 \bar{ \lambda} \le \frac{H_N}{2}$.

\medskip \noindent Now consider the function  
\[ 
w_3:=1- \frac{9}{5} r^\frac{4}{3} + \frac{4}{5} r^3. 
\]  
We show that it is a singular  $H^2(B)-$weak sub-solution of $ (P)_\frac{H_N}{2}$ that is semi-stable. Indeed, we clearly have that  $ 0 \le w_3 \le 1$ a.e. in $B$, $ w_3 \in H_0^2(B)$ and $ \frac{1}{1-w_3} \in L^3(B)$.   To show the stability condition, we consider  $ \psi \in C_c^\infty(B)$ and write
\begin{eqnarray*}
H_N \int_B \frac{ \psi^2}{(1-w_3)^3} &=& 125 H_N \int_B \frac{ \psi^2}{ (9r^\frac{4}{3}-4r^3)^3} \le  125 H_N \sup_{0<r<1} \frac{1}{(9-4r^{ 3-\frac{4}{3}})^3} \int_B \frac{\psi^2}{r^4} \\
&=& H_N \int_B \frac{ \psi^2}{r^4} \le \int_B (\Delta \psi)^2.
\end{eqnarray*} 
An easy computation shows that 
\begin{eqnarray*}
\frac{H_N}{2(1-w_3)^2} - \Delta^2 w_3 &=& \frac{ 25 H_N}{2 ( 9r^\frac{4}{3}-4r^3)^2} - \frac{ 9 \bar{\lambda}}{5 r^\frac{8}{3}} - \frac{12}{5}\frac{N^2-1}{r}\\
&=& \frac{25 N^2 (N-4)^2 }{32 ( 9 r^\frac{4}{3}-4r^3)^2} - \frac{8 ( N-\frac{2}{3})(N-\frac{8}{3})}{ 5 r^\frac{8}{3}}  - \frac{12}{5}\frac{N^2-1}{r}
\end{eqnarray*} 
and by using Maple,  one can verify that this final quantity is nonnegative on $ (0,1)$, whenever  $ 17 \le N \le 30$, hence $ w_3 $ is a subsolution of $ (P)_\frac{H_N}{2}$. If now, $\frac{H_N}{2}<\lambda^*$, then Lemma \ref{poo} would imply that the minimal solution $ u_\lambda$ is larger than  $w_3$  and hence is singular for $ \frac{H_N}{2} < \lambda < \lambda^*$, which is a contradiction.
\end{proof}

\begin{remark} \rm We believe that the extremal solution is singular for all $N\geq 9$ and for that, one needs to construct again for the remaining cases $9\leq N\leq 16$ (and $\frac{H_N}{2} < \lambda^*$),  a singular $H^2(B)-$weak sub-solution of $ (P)_\frac{H_N}{2}$ that is semi-stable. However, one can show that (at least for $N=9$) such a sub-solution cannot be obtained by simply perturbing $\bar u$ with a function of the form  $\phi_0=\frac{4}{3} \beta r^\alpha(1-r^\beta)$.\\
\noindent The construction of such a sub-solution for the remaining cases, i.e. when $9\leq N\leq 16$ and $\frac{H_N}{2} < \lambda^*$, will therefore very likely require a computer assisted proof, that we leave open to the interested reader. 
\end{remark}

\end{document}